\documentclass[12pt,journal,onecolumn ]{IEEEtran}

\usepackage{amssymb}
\usepackage{physics}
\usepackage{graphicx}  % remove 'demo' option for your real document

 \usepackage{cuted} % for long eqns in two column papers
 \usepackage{amsmath}
\usepackage{amssymb}
\usepackage{url}
\usepackage{hyperref} 

\usepackage{tikz}
\usetikzlibrary{matrix,positioning,calc}

\usepackage{setspace}
\usepackage{comment}
\usepackage{amsthm}

\font\Bigmath=cmsy10 scaled \magstep2
\font\bigmath=cmsy10 scaled \magstep1
\def\diamondplus{\mathrel{%
  \ooalign{\raise.29ex\hbox{$\scriptscriptstyle+$}\cr\hss$\diamond$\hss}}}
\def\diamondplustwo{\mathrel{%
  \ooalign{$+$\cr\hss\lower.255ex\hbox{\Bigmath\char5}\hss}}}
\def\diamondplusthree{\mathrel{%
  \ooalign{$\scriptstyle+$\cr\hss\lower.29ex\hbox{\bigmath\char5}\hss}}}

\renewcommand{\Tr}{\operatorname{Tr}}
\renewcommand{\tr}{\operatorname{tr}}

\newcommand{\diag}{\operatorname{diag}}

\newcommand{\be}{\begin{equation}}
\newcommand{\ee}{\end{equation}}
\newcommand{\R}{\mathbb R}
\newcommand{\B}{\mathbb B}
\newcommand{\C}{\mathbb C}

\def\GL{\mathop{\rm GL}\nolimits}
\def\Sym{\mathop{\rm Sym}\nolimits}
\def\Spec{\mathop{\rm Spec}\nolimits}
\def\gl{\mathop{\rm gl}\nolimits}

\newtheorem{Theorem}{Theorem}
\newtheorem{Lemma}{Lemma}
\newtheorem{Proposition}{Proposition}
\newtheorem{Corollary}{Corollary}
\newtheorem{Definition}{Definition}
\newtheorem{Remark}{Remark}
\newtheorem{Example}{Example}

\newtheorem{example}{Example}
\AtBeginEnvironment{example}{%
  \pushQED{\qed}%
}
\AtEndEnvironment{example}{\popQED\endexample}

\usepackage{amssymb}
\usepackage{amsmath}
\begin{document}
%%%%%%%%%%%%%%
\title{Guardian maps for continuous-time systems:\\ A Lie-algebraic approach\thanks{Research supported in part by the Israel Science Foundation.}} %% Article title


\author{Eyal Bar-Shalom,   Alexander Ovseevich and Michael Margaliot\footnote{The authors are with the School of ECE, Tel Aviv University, Israel 69978. Correspondence: michaelm@tauex.tau.ac.il}} %% Author name

\maketitle

\iffalse
\affiliation{organization={},%Department and Organization
            addressline={},
            city={},
            postcode={},
            state={},
            country={}}
\fi

%% Abstract

\begin{abstract}
%%%%%%%%%%%%%%%%%%%%%%%%%%%%%%%
Guardian maps  are scalar maps  that vanish when a  matrix or polynomial is on the verge of stability.
  Several guardian maps have been proposed in the literature for Hurwitz stability based on the Kronecker sum, the second lower Schl\"aflian
 matrix, and the  bialternate sum. It is natural to ask if there is a unifying principle for  all these
  maps. Here, we introduce the Lie-algebraic notion of a guardian representation, and show that all the examples above are instances of this unifying idea. 
  We also show that the bialternate sum coincides with the second additive compound.
 % This framework naturally suggests the second additive compound as another guardian map for Hurwitz stability, but we show that the bialternate sum coincides with the second additive compound.
%%%%
\end{abstract}

\begin{IEEEkeywords}
    Compound matrices, Lie algebra  representations, 
    robust stability. 
\end{IEEEkeywords}
%\begin{center}
%\begin{IEEEMSC}
%15B30, 93D09.
%\end{IEEEMSC}
%\end{center}

%% Keywords

%\end{frontmatter}



\section{Introduction}\label{sec1}
 %%%%%%%%%%%5
Guardian  maps are
used to analyze   the stability of  parameter-dependent
matrices
or polynomials
that are typically
associated with dynamic systems~\cite{Guardian1990}\cite[Ch. 17]{Barmish_book}. They have found numerous theoretical and practical applications in 
robust stability analysis of singularly perturbed   systems~\cite{SAYDY1996807}, interval 
systems~\cite{Chen01122012}, parameter dependent
LTI systems~\cite{Rern_et_al1994,Zhang01092006,MUSTAFA1995,vicono1990}, as well as for 
control synthesis in numerous applications (see e.g.~\cite{MUSTAFA1994,Bouazizi_PID_Guardian,launch_via_guardian,PowerSyst_GuardMaps_Ma,CHIHABI2021370}).

Guardian maps can be used to analyze the (generalized)
 stability   of a matrix  
 which entails that its eigenvalues
  lie in some
prescribed domain
of the complex plane. The classical  case is where the domain  of interest is
the open left-half complex plane  (corresponding to continuous-time stability) or
the open unit disk (discrete-time stability). 
%%%%%%%%%%

A   guardian
map  for continuous-time stability is a mapping~$f:\R^{n\times n} \to \R$
  such that  
\be\label{eq:defn}
f(A)=\begin{cases}
\not = 0, & \text{ if } \Re(\lambda_i(A))< 0 \text{ for all } i, \\
=0,& \text{ if } \Re(\lambda_i(A))= 0 \text{ for some } i . 
\end{cases}
\ee
Thus, $f$ provides  an explicit characterization  for the boundary of the set of Hurwitz  stable matrices. 
It is possible to require a weaker condition on the guardian map, namely, that it satisfies 
%%%%
\be\label{eq:defm}
g(A)=\begin{cases}
\not = 0, & \text{ if } \Re(\lambda_i(A))< 0 \text{ for all } i, \\
=0,& \text{ if } A \text{ admits an eigenvalue } i\beta,\; \beta\not =0. 
\end{cases}
\ee
Indeed, 
in this case~$f(A):=\det(A)g(A)$ satisfies~\eqref{eq:defn}, as~$\det(A)$ becomes zero when~$A$ admits a zero eigenvalue.

To demonstrate   one possible
guardian map that satisfies~\eqref{eq:defn} recall that for a matrix~$A\in\R^{n\times n}$  the Kronecker sum of~$A$ with itself
is~$A\oplus A:= A  \otimes I_n+I_n \otimes A$,
where~$\otimes$  
stands for the Kronecker 
product~\cite[Chapter~4]{Horn1991TopicsMatrixAna},
and~$I_n$ is the $n\times n$ identity matrix.
  Mustafa~\cite{MUSTAFA1994} used the guardian map
\[
f_1(A) :=
\det( A  \oplus A )
\]
to derive an elegant formula ensuring   the stability of a closed-loop system composed of a linear system and an integral controller.
Recall that the eigenvalues of~$A\oplus A $ are
the~$n^2$ sums
\[
\lambda_i(A)+\lambda_j(A),\quad i,j\in\{1,\dots,n\}.
\]
In particular, if~$\Re(\lambda_i(A))<0$ for all~$i$ then the same property holds for the eigenvalues of~$A\oplus A$, and if~$A$ has a zero  eigenvalue  or an eigenvalue~$i\beta$, $\beta\not =0$, 
then~$A\oplus A$ has a zero eigenvalue.
This implies that~$f_1$
satisfies~\eqref{eq:defn}.
However, note that~$f_1$ requires the
         computation of the
         determinant of an~$n^2\times n^2$ matrix. Another difficulty associated with~$f_1$ is that if~$A$ admits a block structure (e.g., when studying the stability of a feedback system or a singularly perturbed system) the Kronecker sum destroys
         the structure.
The latter problem can be addressed for structured matrices by using the 
block Kronecker sum (see e.g.~\cite{MUSTAFA010}). 

Several other guardian 
maps that satisfy~\eqref{eq:defn} or~\eqref{eq:defm}  
have been suggested.
These are based on 
the second lower Schl\"aflian
 matrix of~$A$ (see e.g.~\cite{Guardian1990}),
and 
the bialternate sum of~$A$ with itself (see e.g.~\cite{MUSTAFA1995}). 
%
%
\begin{comment}
    has   eigenvalues
\[
\lambda_i(A)+\lambda_j(A),\quad i,j\in\{1,\dots,n\} \text{ with } i\not = j ,
\]
so it can also be used to define a guardian map~\cite{Guardian1990},
yet the  dimensions of~$A_{[2]}$ are~$\frac{n(n+1)}{2} \times \frac{n(n+1)}{2} $, which is an improvement with respect to the dimensions of~$A\oplus A$.

\end{comment}
%
%
\begin{comment}
    
Later Mustafa with coauthors \cite{MUSTAFA1995} studied many other guardian maps of the form $A\mapsto\det\rho(A)$ similar to the above Kronecker double, like the ``Lyapunov double'' $\rho(A):\Sym_n\to\Sym_n$ given by $\rho(A)X= A^TX+XA$, where $X$ runs over space $\Sym_n$ of symmetric matrices on $\R^n$.

\end{comment}
%
%
A natural question is why   these different 
algebraic structures are all 
  suitable  for guardian maps?

  Here, we introduce the new Lie-algebraic 
notion of a \emph{guardian representation}  and show that all 
 known  examples of guardian maps for the set of Hurwitz matrices   are a particular 
 case of this unifying  concept. 
This new notion also suggests that a natural algebraic structure for a guardian map is the second additive compound  of a matrix~\cite{ofir_suff_comp}. However, it turns out that this is not really a new idea, as we  show that the bialternate sum of~$A$ with itself
is just  the second additive compound of~$A$. 
%%%
The main contribution of this work lies in identifying and proving new
linear--algebraic identities among well--known matrix constructions,
thereby unifying them under a common Lie-algebraic framework.

The remainder of this note  is organized as follows. The next section quickly reviews  several topics   that are used later on.
Section~\ref{sec:main}
introduces  the new notion of a guardian representation, and then    shows that all previous algebraic structures used to define a guardian map for Hurwitz matrices (that is, 
the Kronecker product, the second lower Schl\"aflian
 matrix, and 
the bialternate sum)
are all a special case of  this unifying notion. 
There is another important relation between these algebraic notions, based on   a matrix ODE, and for 
 the sake of completeness, we include it in the Appendix.

We   use standard notation. Vectors [matrices] are denoted by small [capital] letters. 
For a matrix~$A$, $A^\top$ denotes the transpose of~$A$.
The eigenvalues of a matrix~$A\in\R^{n\times n}$ are denoted~$\lambda_1(A),\dots,\lambda_n(A)$. 
%%%%%%
In the context of  Lie groups and Lie algebras
we  sometimes also denote matrices  by small letters, as this  is the standard notation in this field.

\section{Preliminaries}

This section  briefly reviews  known results on compound matrices, 
 Lie groups, Lie algebras and representations 
 that are used later on.
 
%%%%%%%%%%%%%%%%%%%%%%%%%%%
\subsection{Compound Matrices}
%%%%%%%%%%%%%%%%%%%%%%%%%%%
Compound matrices are a tool from
multilinear algebra that has  recently found many applications in systems and control theory~\cite{comp_long_tutorial}, 
in particular in the context 
of~$k$-contraction~\cite{kordercont,Angeli2025smallgain,ofir2021sufficient,Dalin2022Duality_kcont,comp_via_kronecker},
$\alpha$-contraction~\cite{wu2020generalization},
and~$k$-cooperative systems~\cite{Eyal_k_posi,KATZ2025113651}. Here we  review their spectral properties.

%%%%%%%%%%%%%%%%%%%%%%%%%%%
%%%%%%%%%%%%%%%%%%%%%%%%%%%

%%%%%%%%%%%%%%%%%%%%%%%%%%%
\begin{Definition}\label{def:MC}
%%%%%%%%%%%%%%%%%%%%%%%%%%%
%
Let~$A\in\R^{n\times m}$. Fix~$k\in \{1,\dots,\min\{n,m\}  \}$. The~$k$-multiplicative compound matrix of~$A$, denoted~$A^{(k)}$, is the~$\binom{n}{k}\times\binom{m}{k}$ matrix which consists of all the~$k$-minors of~$A$,
ordered lexicographically.
%
%%%%%%%%%%%%%%%%%%%%%%%%%%%
\end{Definition}
%%%%%%%%%%%%%%%%%%%%%%%%%%%
%
%%%
%

For example, if $n=2,m=3$ i.e.,
$
A=\begin{bmatrix}
    a_{11} & a_{12} & a_{13} \\
    a_{21} & a_{22} & a_{23} \\
\end{bmatrix}
$
then~$A^{(2)}$ is the $1\times 3$ matrix: 
\[
A^{(2)}=\begin{bmatrix}
   \det  \begin{bmatrix}
   a_{11}& a_{12}\\
   a_{21}&a_{22}
\end{bmatrix}
&
%%%
 \det  \begin{bmatrix}
   a_{11}& a_{13}\\
   a_{21}&a_{23}
\end{bmatrix}
&
%%%
\det  \begin{bmatrix}
   a_{12}& a_{13}\\
   a_{22}&a_{23}
\end{bmatrix}
%%%%%%%
\end{bmatrix}.
\]

Definition~\ref{def:MC}
implies that
$A^{(1)}=A$,  and if~$A\in\R^{n\times n}$ then~$A^{(n)}=\det (A)$.
In addition,~$(A^{(k)})^\top = (A^\top)^{(k)}$. 

 %%%%%%
If~$D\in\R^{n\times n}$ is diagonal with eigenvalues $\lambda_1,\dots,\lambda_n$, then Definition~\ref{def:MC} implies that~$D^{(k)}$ is also diagonal, with
$\binom{n}{k}$
eigenvalues
$\prod_{i=1}^{k} \lambda_i, (\prod_{i=1}^{k-1}\lambda_i) \lambda_{k+1},\dots,
\prod_{i=n-k+1}^{n} \lambda_i$.
In particular,
the $k$-multiplicative compound of the~$n\times n$ identity matrix is~$(I_n)^{(k)}:=I_r$, with~$r:=\binom{n}{k}$.

More generally,
the~$k$-multiplicative compound of a square matrix~$A\in\R^{n\times n}$ has a ``multiplicative''  spectral property. 
If $A$ has eigenvalues 
$ \lambda_1,\dots,\lambda_n$
then the~$\binom{n}{k}$ eigenvalues of~$A^{(k)}$ are the products:
\[
\lambda_{i_1}\dots\lambda_{i_k}, \quad 1\leq i_1<\dots< i_k\leq n.
\]

The term ``multiplicative compound'' is justified by the
Cauchy-Binet theorem which asserts that for any~$A\in\R^{n\times m}$, $B\in\R^{m\times p}$ and~$k\in
\{1,\dots,\min\{n,m,p\}\}$, we have 
\be\label{eq:cbin}
(AB)^{(k)}=A^{(k)} B^{(k)}.
\ee
For~$A,B\in\R^{n\times n}$ and~$k=n$,~\eqref{eq:cbin} reduces to~$\det(AB)=\det(A)\det(B)$.

%
%%%%%%%%%%%%%%%%%%%%%%%%%%%%%%%%%%%%%%%%%%%%%%%%%%%%%%%%%
%

%%%%%%%%%%%%%%%%%%%%%%%%%%%
%%%%%%%%%%%%%%%%%%%%%%%%%%%

%%%%%%%%%%%%%%%%%%%%%%%%%%%
\begin{Definition}\label{def:AC}
%%%%%%%%%%%%%%%%%%%%%%%%%%%
%
Let $A \in \C^{n\times n}$. The~$k$-additive compound matrix
of $A$, denoted $A^{[k]}$, is the $\binom{n}{k} \times \binom{n}{k}$
matrix defined by
\begin{align*}
    A^{[k]} :=
    \frac{d}{d\varepsilon}
    (I_n + \varepsilon A)^{(k)} \mid_{\varepsilon = 0}.
\end{align*}
%
%%%%%%%%%%%%%%%%%%%%%%%%%%%
\end{Definition}
%%%%%%%%%%%%%%%%%%%%%%%%%%%
%
This  implies that~$A^{[1]}=A$ and~$A^{[n]}=\tr(A)$.
%%%%%
The term additive compound is justified by the fact that
  for any~$A,B\in\R^{n\times n}$ and any~$k\in\{1,\dots,n\}$, we have  $(A + B)^{[k]} = A^{[k]} + B^{[k]}$.
%
%%%%
%

The~$k$-additive compound  satisfies an ``additive''  spectral property:
If~$A\in\R^{n\times n}$ has  eigenvalues~$ \lambda_1,\dots,\lambda_n$ 
then the $\binom{n}{k}$ eigenvalues 
of~$A^{[k]}$  are the sums 
\[
\lambda_{i_1}+\dots+\lambda_{i_k}, \quad 1\leq i_1<\dots< i_k\leq n.
\]

%%%%%%%%%%%%%%%%%%%%%%%%%
\subsection{Lie groups, Lie algebras and representations}
%%%%%%%%%%%%%%%%%%%%%%%
Let~$\text{GL}(n, \mathbb{R})$
denote  the set of
$n \times n$ invertible matrices with real entries.
We consider only Lie groups of (real) matrices~\cite{Hall_book_Lie_algebras}, that is,  sets of invertible matrices that are closed under inversion and multiplication, and possess a compatible structure of a smooth manifold.

Associated with each Lie group is its Lie algebra, which is %can be thought of as
the tangent space to the Lie group at the identity element~$I_n$,  equipped with the Lie bracket~$[\cdot,\cdot]:\R^{n\times n}\times \R^{n\times n}\to\R^{n\times n}$ defined by
\be\label{eq:lie_bracket}
[A,B]:= AB-BA,
\ee
i.e., the commutator of~$A$ and~$B$.
 The Lie algebra captures the ``first-order infinitesimal'' structure of the Lie group, and it   describes completely
 how the group looks locally near the identity.

Let
  $\gl(n,\R)$ denote  the
Lie algebra   of all real $n\times n$ matrices, with the Lie bracket~\eqref{eq:lie_bracket}.
A  linear  map~$\rho:\gl(n,\R)\to\gl(m,\R) $
is a Lie-algebra representation if
 \be\label{eq:def_rep}
 \rho([A,B])=[\rho(A),\rho(B)]
 \ee
 for all~$A,B \in \gl(n,\R)$. In other words,~$\rho$ preserves the Lie bracket. 

\begin{example}\label{exa:kron_sum_is_repres}
 Consider~$\rho:gl(n,\R ) \to gl(n^2,\R) $  given by~$\rho(A)=A\oplus A$.  Using the mixed-product rule of the Kronecker product, that is,  
\[
(X\otimes Y)(Z\otimes W) = (XZ) \otimes (YW),
\]
(see~\cite[Chapter~4]{Horn1991TopicsMatrixAna})
and the  bilinearity of $\otimes$ yields 
\begin{align*}
[\rho(A),\rho(B)]
&= (A\otimes I_n + I_n\otimes A)(B\otimes I_n + I_n\otimes B) \\
& - (B\otimes I_n + I_n\otimes B)(A\otimes I_n + I_n\otimes A) \\
&= (AB\otimes I_n + A\otimes B + B\otimes A + I_n\otimes AB) \\
&  - (BA\otimes I_n + B\otimes A + A\otimes B + I_n\otimes BA) \\
&= (AB-BA)\otimes I_n \;+\; I_n\otimes(AB-BA) \\
&= [A,B]\otimes I_n \;+\; I_n\otimes [A,B] \\
&= \rho([A,B]), 
\end{align*}
so $\rho$ is  Lie algebra representation.
\end{example}

For our purposes, it is    important to recognize that a map~$\rho$ is a  Lie algebra representation without explicit computations as in Example~\ref{exa:kron_sum_is_repres}.
One advantage of the relation between Lie groups and Lie algebras is  that it 
allows to verify that a map~$\rho:\gl(n,\R)\to\gl(m,\R) $
is indeed a Lie  algebra representation.
Given~$R:\GL(n,\R) \to \GL(m,\R)$, 
define~$\rho:\gl(n,\R)\to \gl(m,\R)$  by
\begin{equation}\label{eq:def_rho}
    \rho (A):=\left.\frac{d}{dt}\right\vert_{t=0}R(e^{At}),
\end{equation}
If $R$ is  a  group representation, that is,  $R$ is a continuous map\footnote{The continuity here is equivalent to other regularity conditions, like differentiability and analyticity~
\cite[Ch.~5]{serre}).} satisfying 
\be\label{eq:gr_rep}
R(gh)=R(g)R(h) \text{ for all } g,h\in\GL(n,\R)
\ee
then~$\rho$ is a Lie algebra representation
(see e.g.~\cite{Hall_book_Lie_algebras}).

\begin{example}\label{exa:oplus_lie_aleg}
    As a simple example of these notions, that already points to the main ideas  in this paper, consider
the map~$R:\GL(n,\R)\to\GL(n^2,\R)$ defined by
$R(g): =g\otimes g$, where~$\otimes $ is the Kronecker product.
By the mixed product property of the Kronecker product,
\begin{align*}
    R(gh)&= (gh)\otimes(gh) \\
         &=   (g\otimes g)(h\otimes h)\\
         &= R(g)R(h),
\end{align*}
so~$R$ is a representation   of the group~$\GL(n,\R)$.
Let
\begin{align*}
    \rho(A) &:= \left.\frac{d}{dt}\right\vert_{t=0}R(e^{At})\\
   &=  \left.\frac{d}{dt}\right\vert_{t=0} ( e^{At}\otimes e^{At})\\
   &=A  \otimes I_n+I_n \otimes A\\
   &=A\oplus A,
\end{align*}
 so we conclude that the Kronecker sum is a Lie algebra  representation.
%%%%%%%%%%%%%%%
\end{example}

\begin{Remark}\label{rem:gropu_I}
%%%%%%    
Suppose that~$R:\GL(n,\R) \to \GL(m,\R)$ is a group representation. Then~\eqref{eq:gr_rep} implies that~$R$ maps~$I_n$ to~$I_m$.
If~$h\in\R^{n\times n}$ is nonsingular then
\begin{align*}
    I_m&=R(I_n)\\
    &=R(hh^{-1})\\
    &=R(h)R(h^{-1}),
\end{align*}
so~$(R(h))^{-1}=R(h^{-1})$. 
This implies that for any~$g\in\GL(n,\R)$ we have
\[
R(hgh^{-1})=R(h) R(g) (R(h))^{-1},
\]
%%%%
so in order to understand the spectral structure of~$R(g)$ it is sufficient 
to analyze the case where~$g$ is diagonal. 
%%%%
\end{Remark}

\section{Main results}\label{sec:main}
%%%%%%%%%%%%%%%%%%%%%%%%%%%%%%%%%%%%%
%
This section includes our main results. We begin by 
 introducing
a new definition.
%%%%%%%%%%%%%%%%%%
\begin{Definition}\label{def:guar_rep}
%%%%%%%%%%%%%%%%%%%%%%
A map~$\rho:\gl(n,\R)\to\gl(m,\R) $  is called a  \emph{guardian representation} 
if it is a Lie algebra representation and 
$
\det(\rho(A)) $
satisfies~\eqref{eq:defn} or~\eqref{eq:defm}.
%%%%%%%%%%%%%%%%%%%%%%%
\end{Definition}

In other words,~$\rho$ is a Lie algebra representation, and its determinant
can be used as a guardian map.

\begin{Proposition}\label{prop:2comp}
%%%%%%%%%%%%%%%%%%%%%%%%%%%%%%%%%%%%%%%
   The map~$\rho: \R^{n\times n} \to \R^{\binom{n}{2}\times\binom{n}{2}} $  defined  by~$\rho(A):=A^{[2]}$  is a guardian representation.
%%%%%%%%%%%%%%%
\end{Proposition}
\begin{proof}
The~$2$-additive compound of~$A$  satisfies  
\be\label{eq:defadd}
A^{[2]}:=\frac{d}{dt}\mid_{t=0}(e^{At})^{(2)} 
\ee
(see Definition~\ref{def:AC}). 
%%%%%% 
The eigenvalues of~$A^{[2]}$ are the $\binom{n}{2}$
sums
\[
\lambda_i(A)+\lambda_j(A), \quad 1\leq i<j\leq n 
\]
(see, e.g., \cite{comp_long_tutorial}).
This implies that~$ \det(A^{[2]})$ satisfies the condition in Eq.~\eqref{eq:defm}. 
Define~$R:\GL(n,\R) \to \GL(\binom{n}{2},\R ) $ by~$R(g):=A^{(2)}$. Eq.~\eqref{eq:cbin} implies that~$R$ is a group representation and thus~\eqref{eq:defadd}
 implies that~$\rho(A)$ is a Lie algebra representation  (see also~\cite{compound_lies}, \cite{serre}), and thus a guardian representation. 
%%% 
\end{proof}

Using~$\det(A^{[2]})$ as a guardian map for  structured  systems requires an expression for  the second additive compound of a block matrix. These are also important for the analysis of  2-contraction~\cite{kordercont} of non-linear feedback systems, and explicit expressions  have recently been derived in~\cite{ofir_suff_comp,Angeli2025smallgain}.
 %%%%%%%%

The next result characterizes basic transformations that preserve guardian representations.
\begin{Proposition}\label{prop:2}
Suppose that~$\rho:\gl(n,\R)\to\gl(m,\R) $  is  a  guardian representation.
Then the following mappings are also
guardian representations:
\begin{enumerate}
    \item $\tilde \rho (A):=T\rho(A)T^{-1}$, where $T\in\R^{m\times m}$ is an invertible matrix.
    In other words, a  representation that is isomorphic   to a guardian  representation  is a guardian representation.
    \item $\hat \rho(A):=(\rho(-A))^\top$. In other words, a  representation that is contragradient    to a guardian  representation  is a guardian representation. 
\end{enumerate}
%%%%%%%
\end{Proposition}
\begin{proof}
%%%%
Eq.~\eqref{eq:lie_bracket} implies that if~$H $ is invertible then
\begin{align*}
[HAH^{-1},HBH^{-1}]&=H  A H^{-1} H BH^{-1} -HBH^{-1} HA H^{-1}\\
&=H[A,B]H^{-1},
\end{align*}
and thus if~$\rho$ is a Lie-algebra representation then  so is~$\tilde \rho (A):=T\rho(A)T^{-1}$. Since the determinant is invariant under a similarity transformation, this implies that if~$\rho$ is a guardian  representation
then so is~$\tilde \rho$.
%%%

To prove that~$\hat \rho$ is a Lie-algebra representation, note that
\begin{align*}
  \hat \rho([A,B] )& = ( \rho (-[A,B])  )^\top \\
 % &=  ( \rho ([B,A])  )^\top \\
%  &= ( [\rho(B),\rho(A) ]  )^\top \\
  &= -( [\rho(A),\rho(B) ]  )^\top \\
  &=[\hat \rho(A),\hat \rho (B)] ,
\end{align*}
%%%
so~$\hat \rho$ is also a representation. Since the eigenvalues of~$-A^\top$ are~$-\lambda_i(A)$, $\hat \rho$ also satisfies the guardian map condition.
%%%%%%%%%%%%
\end{proof}

\subsection{Known guardian maps are guardian representations}
\label{sec:previous}
%%%%%%%%%%%%%%%
We already showed that~$\rho(A):=A\oplus A$ preserves the Lie bracket, and thus it is a guardian representation. 
More generally, we now show that all the guardian maps suggested in the past for guarding Hurwitz matrices  are in fact
guardian representations, 
and thus the  new
notion in Definition~\ref{def:guar_rep}
provides a unifying  framework
for  all those  guardian maps.

%%%%%
\subsubsection{The lower  Schl\"aflian
  matrix of order 2}
%%%%%
%
Let~$z\in\R^n$ and fix an integer~$p\geq 1$. Let $s_p(z)$ denote the vector  formed by listing in a lexicographical ordering all the  terms
\[
z_1^{p_1}\dots z_n^{p_n} \text{ with } p_i\in\{0,1,\dots,n\} ,\; \sum_{i=1}^n p_i=p.  
\]
Thus, the dimension of~$s_p(z)$ is~$r_p(n):=\binom{n+p-1}{p}$. 
For example, $s_1(z)=z$, and
if~$n=2$   then
\[
s_2(z)=\begin{bmatrix}
    z_1^2& z_1 z_2& z_2^2
\end{bmatrix}^\top.
\]

Let~$A\in\R^{n\times n}$.
The associated upper 
Schl\"aflian
  matrix
of order~$p$, denoted~$U_p(A)$,  is the $r_p(n)\times r_p(n)$ matrix satisfying
\be\label{eq:pmat}
s_p(Az)=U_p(A)s_p(z) \text{ for all } z\in\R^n. 
\ee
%%%%%%%%%%%%%%%%%%%%
For example, for~$n=p=2$, 
Eq.~\eqref{eq:pmat} becomes
\[
 \begin{bmatrix}
   (a_{11} z_1 +a_{12} z_2 )^2\\ (a_{11} z_1 +a_{12} z_2) (a_{21} z_1 +a_{22} z_2 )\\ (a_{21} z_1 +a_{22} z_2 )^2
\end{bmatrix} =  U_2(A) 
\begin{bmatrix}
z_1^2\\ z_1 z_2\\ z_2^2
\end{bmatrix} , 
\]
so~$U_2(A)=\begin{bmatrix}
a_{11}^2  & 2a_{11}a_{12} & a_{12}^2\\
%%%%%
a_{11}a_{21}&a_{11}a_{22}+a_{12}a_{21}&a_{12}a_{22}\\
%%%%
a_{21}^2& 2a_{21}a_{22}& a_{22}^2
\end{bmatrix}$. As another example, suppose that~$A=\diag(\lambda_1,\dots,\lambda_n)$. Then~\eqref{eq:pmat} gives
\begin{align*}
U_p(A)s_p(z)&= s_p(\begin{bmatrix}
    \lambda_1 z_1&\dots&\lambda_n z_n
\end{bmatrix}^\top)   ,
\end{align*}
so the eigenvalues of~$U_p(A)$ are
\be\label{eq:prot}
\lambda_1^{p_1}\dots \lambda_n^{p_n} \text{ with } p_i\in\{0,1,\dots,n\}  ,\; \sum_{i=1}^n p_i=p.  
\ee
%%%%%%%%%
For example, for $n=2$ and $A=\diag(\lambda_1,\lambda_2)$,
we have
$U_{2} (A) = 
\diag(\lambda_{1}^{2}, \lambda_1 \lambda_2, \lambda_{2}^{2})$.

Given~$A\in\R^{n\times n}$,
the associated lower  
Schl\"aflian
 matrix
of order~$p$, denoted~$L_p(A)$,  is the~$r_p(n)\times r_p(n)$ matrix defined 
as follows. If
\be\label{eq:simpax}
\dot z(t)=A z(t)  
\ee
then
\be\label{eq:deflow}
\frac{d}{dt} s_p(z(t))= L_p(A) s_p(z(t)). 
\ee
The transformation from~\eqref{eq:simpax} to~\eqref{eq:deflow} is sometimes called the  power transformation~\cite{BARKIN1983303}.   Fuller~\cite{Fuller1968} refers to~$L_2(A)$ as the Lyapunov matrix associated with~$A$. 
\begin{Example}
    For~$n=p=2$, Eq.~\eqref{eq:deflow}
becomes
\[
\begin{bmatrix}
2z_1 (a_{11}z_1+a_{12}z_2)\\
z_1 (a_{21}z_1+a_{22}z_2)+
z_2 (a_{11}z_1+a_{12}z_2)\\
2z_2 (a_{21}z_1+a_{22}z_2)
\end{bmatrix}=L_2(A)
\begin{bmatrix}
    z_1^2\\z_1z_2\\z_2^2
\end{bmatrix},
\]
so
\be\label{eq:las2}
L_2(A)=\begin{bmatrix}
    2a_{11} & 2a_{12}&0 \\
    a_{21} & a_{11}+a_{22} & a_{12}\\
    0& 2a_{21} &2a_{22}
\end{bmatrix}.
\ee
\end{Example}

Eq.~\eqref{eq:deflow}
implies that
\begin{align*}
    L_p(A)s_p(z(t)) & = \frac{d}{dt} s_p(e^{At}z_0)\\
    &=\frac{d}{dt} U_p(e^{At})s_p(z_0),
\end{align*}
and setting~$t=0$ gives
\be\label{eq:nedfork}
    L_p(A)  =
    \frac{d}{dt} \mid_{t=0} U_p(e^{At}).
%%%
\ee
Thus, the relation between~$U_p(A)$ and~$L_p(A)$ is an instance of the relation between representations of Lie groups and Lie algebras.

Combining  \eqref{eq:prot} and~\eqref{eq:nedfork}   implies that the eigenvalues of~$L_p(A) $ are 
\[
\lambda_1 (A){p_1}+\dots+ \lambda_n (A) {p_n} \text{ with } p_i\in\{0,1,\dots,n\} ,\; \sum_{i=1}^n p_i=p.  
\]
For example, for~$n=p=2$ and~$A=\diag(\lambda_{1},\lambda_{2})$, Eq.~\eqref{eq:las2} implies 
that~$L_2(A)=\diag(2\lambda_{1},\lambda_{1}+\lambda_{2},2\lambda_{2})$. 

We conclude in particular   
that~$\det(L_2(A))$ satisfies~\eqref{eq:defn}, so~$\rho (A)=L_2(A)$ is a guardian map. 
Note that~$L_2(A) $ has dimensions~$\binom{n+1}{2}\times \binom{n+1}{2}$, whereas
the matrix~$A\oplus A$ has dimensions~$n^2\times n^2$. 

\begin{Proposition}\label{prop:lower}
 The map $\rho(A):=L_2(A)$ is a guardian representation.
\end{Proposition}
    
\begin{proof}
%%%%%%
Fix~$A,B\in\R^{n\times n} $. By~\eqref{eq:pmat},
$s_p(ABz)=U_p(AB)s_p(z)$, and also
\begin{align*}
s_p(ABz)&=s_p(A(Bz))\\
&=U_p(A)s_p(Bz)\\&
=U_p(A)U_p(B)s_p(z),
%%%%%
\end{align*}
so~$U_p(AB)=U_p(A)U_p(B)$. 
This implies in particular
that the map~$R:GL(n,\R)\to G(r_2(n),\R)$ defined by~$R(g):=U_2(g)$
is a group
representation. Now 
\begin{align*}
 %%%%%%%%%%%   
    \frac{d}{dt} \mid_{t=0} R(e^{At}) 
    &= 
    \frac{d}{dt} \mid_{t=0} U_2(e^{At}) \\
    &=
   L_2(A),
\end{align*}
where the last equation follows from~\eqref{eq:nedfork},  so~$L_2(A)$ is a Lie algebra representation, and this completes the proof of Prop.~\ref{prop:lower}.
%%%%%%
\end{proof}

 \subsubsection{The bialternate sum}
 %%%%%%%%%%
The bialternate sum~\cite{Fuller1968} of a matrix~$A\in\R^{n\times n}$ with itself,
denoted~$A \diamondplus A$, 
is an~$\frac{n(n-1)}{2}\times \frac{n(n-1)}{2}$ matrix
defined as follows (see, e.g.~\cite{MUSTAFA1995}). 
Denote by~$L(n)$ the ordered list of pairs
\[
\{(2,1),(3,1),\dots,(n,1),(3,2),(4,2),\dots,(n,2),\dots,(n-1,n-2),(n,n-2),(n,n-1)\}.
\]
Fix indexes~$x,y\in\{1,\dots, \frac{n(n-1)}{2}\}$. Let~$(p,q)$ [$(r,s)$] be the pair in position~$x$ [$y$] in the list~$L(n)$. Then entry~$(x,y)$ of~$A\diamondplus A$ is
\be\label{eq:xylist}
a_{pr}\delta_{qs}+a_{qs} \delta_{pr}-a_{ps}\delta_{qr}-a_{qr}\delta_{ps},
\ee
where~$\delta_{ij}=1$ if~$i=j$ and zero, otherwise. 

For example, when~$n=2$ the list~$L(2)$ includes the single pair
$(2,1)
$, 
and~$A\diamondplus A $ is the scalar 
corresponding to~$(p,q)=(r,s)=(2,1)$, that is, 
\[
a_{22}\delta_{11}+a_{11} \delta_{22}-a_{21}\delta_{12}-a_{12}\delta_{21}
=a_{22}+a_{11}.
\]
 
The eigenvalues of~$A\diamondplus A$ are the sums
\[
\lambda_i(A)+\lambda_j(A),\quad 1\leq i<j\leq n,
\]
and thus~$\det(A\diamondplus A)$ is a guardian map satisfying~\eqref{eq:defm}. 

To the best of our knowledge, the   next
result is new.
%%%%
\begin{Proposition}\label{prop:new}
%%%%%%%
Fix~$A\in\R^{n\times n}$. Then 
     \[
    A\diamondplus A=  A^{[2]} .
    \]
\end{Proposition}
\begin{proof}
%%%%
We begin by recalling an explicit expression for the entries of~$A^{[2]}$. 
Since~$A^{(2)} $ (and thus~$A^{[2]})$ is based on the 2-minors of~$A$, it is natural to index the entries of~$A^{[2]}$ using a pair of row indexes and a pair of  column indexes in~$A$. 
Let~$\tilde L(n)$ denote the list of lexicographically ordered pairs
\[
\{(1,2),(1,3),\dots,(1,n),(2,3),\dots,(2,n),\dots,(n-1,n)\}.
\]
(Note that~$L(n)$ and~$\tilde L(n)$ are identical, only with an opposite ordering inside each pair.)
Fix two pairs~$(i_1,i_2)$ and~$(j_1,j_2)$ from~$\tilde L(n)$. It is known~\cite{schwarz1970} that the entry of~$A^{[2]}$ corresponding to rows [columns] 
$(i_1,i_2)$
[$(j_1,j_2)$] of~$A$ is
\[
\begin{cases} 
a_{i_1i_1}+a_{i_2i_2}, & \text{if } i_1=j_1 \text{ and } i_2=j_2,\\
%%%%
(-1)^{\ell+m} a_{i_\ell j_m}, & \text{if } (i_1,i_2) \text { and }  (j_1,j_2) \text{ share one common index and } i_\ell\not =j_m, \\
%%%
0,& \text{otherwise}. 
\end{cases}
\]
Note that we can write this as
\be
\delta_{i_1 j_1} a_{i_2 j_2}  
+\delta_{i_2 j_2} a_{ i_1 j_1}
-\delta_{i_1 j_2 } a_{ i_2 j_1} 
-\delta_{i_2 j_1} a_{i_1 j2 }. 
\ee
Comparing this with~\eqref{eq:xylist} and noting that the pairs in~$L(n)$ are ordered in the opposite ordering completes the proof of Prop.~\ref{prop:new}.
%%%%%
\end{proof}

\begin{Corollary}
    The mapping~$\rho(A)=\det(A \diamondplus A)$ is a guardian representation.
\end{Corollary}
\begin{proof}
    This follows immediately
from the result in Props.~\ref{prop:2comp} and~\ref{prop:new}.
%%%%%
\end{proof}

%%%%%%%%%%%%%%%%%
\section{Conclusion}
%%%%%%%%%%%%%%%%%%%%%%%%
Guardian maps play an important role in  stability
analysis and    robust control theory. We showed that all known guardian maps for   continuous-time Hurwitz stability
follow from a unifying Lie-algebraic construction.
Roughly speaking,  in order for   $\det(\rho(A))$ to be guardian map it is necessary that the spectrum of~$\rho(A)$ includes (some) sums in the
form~$\lambda_i+\lambda_j$ of eigenvalues of $A$. 
We showed that this follows from the fact that for some
Lie group 
representation $R$ the spectrum of $ R(B)$,
$B\in \GL_n$, includes (some) products~$\mu_i\mu_j$ of eigenvalues of~$B$.  In this respect, it is interesting to note that
in describing the bialternate sum Fuller~\cite{Fuller1968} states: ``... it is appropriate  to begin by considering a matrix whose characteristic roots are products (rather than sums) of the characteristic roots of a matrix~$A$''.

%%%%%%%%%%%%%%%%%%%%%%%%%%%%%%%%%%%%%%%%%%%%%%%%%%
\section*{Appendix: Guardian maps via a  matrix differential equation}
%%%%%%%%%%%%%%%%%%%%%%%%%%%%%%%%%%%%%%%%%%%%
%%%%%%%%%%%%%%%%%%%%%%%%%
In this appendix, we   review a another important relation between the Kronecker sum, the second additive compound, and the lower Schl\"aflian
 matrix of order 2 that is based on a matrix ODE.
This relation  is known, but often appears without proofs (see e.g.~\cite{Angeli2025smallgain,Brockett_Roger}), so we also include the proofs.  

Fix~$A\in\R^{n\times n}$, and
consider the matrix differential  equation
    \be\label{eq:adotax}
    \frac{d}{dt} X (t)  =A X(t)+X (t)A^\top,
    \ee
with initial condition~$X(0)\in\R^{n\times n}$.
It is straightforward to verify that
the solution is
\be\label{eq:sol_matrix_ODE}
X(t)=e^{At}X(0)e^{A^\top t}.
\ee

For a matrix~$B\in\R^{n\times n}$, let~$
\operatorname{vec}(B)\in\R^{n^2}$ 
denote the column vector obtained by stacking the rows of~$B$ one after the other. For example,
 for~$A=\begin{bmatrix}
     a_{11}&a_{12}\\a_{21}&a_{22}
 \end{bmatrix}$, we have
 \[
\operatorname{vec}(A)=\begin{bmatrix}
    a_{11}&a_{12}&a_{21}&a_{22}
\end{bmatrix}^\top.
\]
Eq.~\eqref{eq:adotax} yields     
\begin{align*}
\frac{d}{dt}\operatorname{vec}( X(t)) 
&= \operatorname{vec}(AX(t)) + \operatorname{vec}(X(t)A^\top)\\
&= (I_n\otimes A)\operatorname{vec}(X(t)) + (A\otimes I_n)\operatorname{vec}(X(t))\\
&=(A\oplus A) \operatorname{vec}(X(t)).
\end{align*}
%%%
In other words, 
the dynamics of all the entries of~$X(t)$ in vectorized form  is governed by~$A\oplus A$.

We now show that if we specialize the matrix ODE to the case of symmetric [skew-symmetric] matrices then the dynamics of~$X(t)$ is governed by $L_2(A)$ [$A^{[2]}]$.

 \subsection{The symmetric case}
 %%%%%%%%%%%
 If~$X\in\R^{n\times n}$
is symmetric  then it is uniquely defined by its entries on and above its main  diagonal, that is, the~$n(n+1)/2$ scalars
\[
x_{ij} , \quad 1\leq i\leq j\leq n , 
\]
so let
\[
w(X):=\begin{bmatrix}
x_{11}&x_{12}&\dots&x_{1n}&
x_{22}& x_{23}&\dots& x_{2n} &\dots&x_{nn}    
\end{bmatrix}
^\top.
\]

\begin{Proposition}
  \label{prop:symm}
Fix~$A\in\R^{n\times n}$.         Consider the matrix differential  equation
    \eqref{eq:adotax}
    with~$X(0)$ symmetric. Then~$X(t) $ is symmetric for all~$t$, and
    \be\label{eq:symmv}
    \frac{d}{dt} w(X(t))=L_2(A) w(X(t)). 
    \ee
\end{Proposition}

In other words, the solution of the   matrix ODE preserves symmetry, 
and the dynamics of the
$ n(n+1)/2$ independent entries of $X(t)$ is  governed by the  
 lower Schl\"aflian
 matrix of order 2
of $A$.

\begin{example}
Consider the case~$n=2$. Then~\eqref{eq:adotax} becomes
\be\label{eq:exhu}
\frac{d}{dt} 
\begin{bmatrix}
 x_{11}&x_{12} \\
 x_{12}& x_{22}   
\end{bmatrix}=
\begin{bmatrix}
 a_{11}&a_{12} \\
 a_{21}&a_{22}   
\end{bmatrix}\begin{bmatrix}
  x_{11}&x_{12} \\
 x_{12}& x_{22}      
\end{bmatrix} + \begin{bmatrix}
  x_{11}&x_{12} \\
 x_{12}& x_{22}   
\end{bmatrix}
\begin{bmatrix}
 a_{11}&a_{21} \\
 a_{12}&a_{22}   
\end{bmatrix}.
\ee
In this case,~$w(X) =\begin{bmatrix}
    x_{11}&x_{12}&x_{22} 
\end{bmatrix}^\top$,
and~\eqref{eq:exhu}  
gives 
\begin{align*}
%%%
\frac{d}{dt} w(X)&=\begin{bmatrix}
    2a_{11}& 2a_{12} &0\\ 
%%%%
a_{21} &a_{11}+a_{22} & a_{12}\\
%%%
0&2a_{21}&2a_{22}
%%%%%%%%%%%%%%%%
\end{bmatrix}
w(X)\\
%%%%%%%%%%%%%
&=L_2(A)w(X) , 
%%%
\end{align*}
where the last equality  follows from~\eqref{eq:las2}.
%%%%%%%
\end{example}

   \begin{proof}[Proof of Prop.~\ref{prop:symm}]
%%%%%%%%%%%%%%%%%%%%%%%%%%
Every entry of~$w(X)$ is of the form~$x_{pq}$  with~$1\leq p\leq q\leq n$. By~\eqref{eq:adotax}, 
%%%%%%%
\be\label{eq:potr}
\dot x_{pq}= \sum_{i=1}^n a_{pi} x_{iq}+\sum_{i=1}^n x_{pi} a_{qi}.
\ee
We compare this to the entry~$z_{p}z_q$ in the vector~$s_2(z)$, where~$z(t)$ satisfies~\eqref{eq:simpax}. Clearly, 
\begin{align*}
    \frac{d}{dt} (z_{p} z_q)& =  \dot z_p z_q+z_p \dot z_q \\
    &=  \sum_{i=1}^n a_{pi}z_iz_q + 
     \sum_{i=1}^n  a_{qi}z_pz_i . 
\end{align*}
%%%
We conclude that    entry~$z_pz_q$ in~$s_2(z)$  and   entry~$x_{pq}$ in~$w(X)$ satisfy the same differential equation. By~\eqref{eq:deflow}, $\frac{d}{dt} s_2(z)=L_2(A)s_2(z)$, and this completes the proof. 
%%%%%%
%%%%%%%%%%%%%%%%%%%%%%%
\end{proof}

 \subsection{The skew-symmetric case}
 %%%%%%%%%%%
Recall that   $X\in\R^{n\times n}$
is called skew-symmetric if~$X+X^\top=0$.
The diagonal entries of such a matrix are zero, and 
it is uniquely defined by its entries above the main diagonal, that is, the~$n(n-1)/2$ scalars
\[
x_{ij} , \quad 1\leq i<j\leq n.
\]
For example, for~$n=3$ we have 
\[
X=\begin{bmatrix}
    0 &x_{12} &x_{13} \\
    -x_{12}&0  &x_{23} \\
   -x_{13}  &-x_{23} &0 
%%%%%%    
\end{bmatrix}.
\]

For a skew-symmetric matrix~$X$, let
\[
v(X):=\begin{bmatrix}
    x_{1,2}& x_{1,3}& \dots & x_{1,n}& 
    x_{2,3} &x_{2,4}&\dots & x_{2,n} &  \dots & x_{n-1,n}
\end{bmatrix}^\top.
\]
 
\begin{Proposition}
  \label{prop:skew}
Fix~$A\in\R^{n\times n}$.     Consider the matrix ODE 
    \eqref{eq:adotax}
    with~$X(0)$ skew-symmetric. Then~$X(t) $ is skew-symmetric for all~$t$,
    and
    \be\label{eq:vxder}
    \frac{d}{dt} v(X(t))=A^{[2]} v(X(t)). 
    \ee
\end{Proposition}

In other words, the solution of the   matrix ODE preserves skew-symmetry, 
the dynamics of all the entries of~$X(t)$ in vectorized form  is governed by~$A\oplus A$,
and the dynamics of the
$ \binom{n}{2}$ independent entries of $X(t)$ is  governed by the second additive compound $A^{[2]}$
of $A$.

\begin{example}
Consider the case~$n=2$. Then~\eqref{eq:adotax} becomes
\[
\frac{d}{dt} 
\begin{bmatrix}
 0&x_{12} \\
 -x_{12}& 0   
\end{bmatrix}=
\begin{bmatrix}
 a_{11}&a_{12} \\
 a_{21}&a_{22}   
\end{bmatrix}\begin{bmatrix}
 0&x_{12} \\
 -x_{12}& 0   
\end{bmatrix} + \begin{bmatrix}
 0&x_{12} \\
 -x_{12}& 0   
\end{bmatrix}
\begin{bmatrix}
 a_{11}&a_{21} \\
 a_{12}&a_{22}   
\end{bmatrix},
\]
\begin{comment}
    In this case,~$\operatorname{vec}(X) =\begin{bmatrix}
    0&x_{12}&-x_{12} &0
\end{bmatrix}^\top$,
and using the fact that the first and fourth entries of~$\operatorname{vec}(X)$ are zero we see that~\eqref{eq:exhu} can be written as 
\begin{align*}
%%%
\frac{d}{dt}
\operatorname{vec}(X)&=\begin{bmatrix}
    2a_{11}& a_{12} & a_{12}&0\\ 
%%%%
a_{21}&a_{11}+a_{22} &0&a_{12}\\
a_{21}&0&a_{11}+a_{22}&a_{12}\\
0&a_{21}&a_{21}&2a_{22}
%%%%%%%%%%%%%%%%
\end{bmatrix}
\operatorname{vec}(X)\\
&=(A\oplus A)\operatorname{vec}(X).
%%%
\end{align*}
\end{comment}
%
so  $\frac{d}{dt} v(X)= \dot x_{12} = a_{11}x_{12} +x_{12}a_{22}$, that is,
$\frac{d}{dt} v(X)=A^{[2]}v(X)$.
%%%%%%%%
\end{example}

\begin{proof}[Proof of Prop.~\ref{prop:skew}]
%%%%%%%%%%%%
Using~\eqref{eq:sol_matrix_ODE}
 it is easy to verify that if~$X(0)$ is skew-symmetric
then~$X(t)$ is skew-symmetric for all~$t$.
%%%%%%%%%%%%
  
In order to prove~\eqref{eq:vxder} we introduce more notation. Let~$e^1,\dots,e^n \in\R^n$ denote the canonical basis of~$\R^n$. 
A   basis for the space of $n\times n$ skew-symmetric matrices is 
the set of~$\binom{n}{2}$ matrices
\[
S_{ij} := E_{ij} - E_{ji}, \qquad 1\le i<j\le n,
\]
where $E_{ij}:=e^i (e^j)^\top$, that is,  the matrix with a \(1\) in position \((i,j)\) and zeros,  elsewhere.
Let~$L:\R^{\binom{n}{2}}\to\R^{n\times n}$ be the linear operator defined by 
\[
L \left (  \begin{bmatrix} e^i & e^j \end{bmatrix}^{(2)} \right ) =S_{ij}, \text{ for all }1\leq i<j\leq n,
\]
  where~$\begin{bmatrix} e^i & e^j \end{bmatrix}^{(2)}$ is the second multiplicative compound of
of the matrix~$\begin{bmatrix} e^i & e^j \end{bmatrix} \in\R^{n\times 2}$.

      We require an axillary result. 
\begin{Lemma}\label{lem:la2}
%%%%%%%%%%%%
    Let~$A\in\R^{n\times n}$. Then for any~$1\leq i<j\leq n$, we have
    \[
    L\left ( A^{[2]} \begin{bmatrix} e^i & e^j \end{bmatrix}^{(2)}\right  )=
    AL\left (  \begin{bmatrix} e^i & e^j \end{bmatrix}^{(2)} \right )+L \left (  \begin{bmatrix} e^i & e^j \end{bmatrix}^{(2)}\right  ) A^\top.
    \]
\end{Lemma}
\begin{proof}
%%%%%%%%
Note that
\begin{align*}
   L\left ( A^{[2]} \begin{bmatrix} e^i & e^j \end{bmatrix}^{(2)}\right  ) & = L\left (  \begin{bmatrix}A e^i & e^j \end{bmatrix}^{(2)} +\begin{bmatrix}  e^i & Ae^j \end{bmatrix}^{(2)}\right ) \\
   %%%
   &=  L\left(   \begin{bmatrix} \sum_{k=1}^n a_{ki} e^k & e^j \end{bmatrix}^{(2)} +\begin{bmatrix}  e^i & \sum_{k=1}^n a_{kj} e^k \end{bmatrix}^{(2)}\right) \\
   %%%%%%%%%%%%%%
   &= \sum_{k=1}^n a_{ki} L\left ( \begin{bmatrix} e^k & e^j \end{bmatrix}^{(2)}\right ) +
   \sum_{k=1}^n a_{kj} L\left ( \begin{bmatrix} e^i & e^k \end{bmatrix}^{(2)}\right  )\\
%%%%%%%%%
&=  \sum_{k=1}^n a_{ki} \left (e^k (e^j)^\top-e^j (e^k)^\top \right )+\sum_{k=1}^n a_{kj}\left  (e^i (e^k)^\top-e^k (e^i)^\top\right )\\
%%%%%%% 
   &=AS_{ij}+S_{ij}A^\top,
%%%%%%%%%%%%%%
%%%%%%%
\end{align*}
%%%%%%%%%%%%
and this completes the proof of Lemma~\ref{lem:la2}. 
\end{proof}

We can now complete the proof of  Prop.~\ref{prop:skew}. 
 Eq.~\eqref{eq:adotax} can be written as 
\[
\frac{d}{dt} \sum_{i<j} x_{ij}(t) S_{ij}  = A\sum_{i<j} x_{ij}(t) S_{ij}+\sum_{i<j} x_{ij}(t) S_{ij}A^\top,
\]
that is,
\[ \sum_{i<j} \dot x_{ij}(t) L\left (\begin{bmatrix}
    e^i&e^j
\end{bmatrix}^{(2)}\right )  =  \sum_{i<j} x_{ij}(t) A L\left (\begin{bmatrix}
    e^i&e^j
\end{bmatrix}^{(2)}\right )
+\sum_{i<j} x_{ij}(t) L\left (\begin{bmatrix}
    e^i&e^j
\end{bmatrix}^{(2)}\right ) A^\top.
\]
Applying Lemma~\ref{lem:la2} gives
\begin{align*}
%%%
\sum_{i<j} \dot x_{ij}(t) L\left(\begin{bmatrix}
    e^i&e^j
\end{bmatrix}^{(2)}\right ) = \sum_{i<j}    L\left( A^{[2]} \begin{bmatrix} e^i & e^j \end{bmatrix}^{(2)} \right )x_{ij}(t)  , 
%%%
\end{align*}
and this completes  the proof of Prop.~\ref{prop:skew}.
\end{proof}

\end{document}